\documentclass[11pt]{amsart}

\usepackage{amsmath}
\usepackage{amstext}
\usepackage{amssymb}
\usepackage{amsthm}
\usepackage{enumerate}
\usepackage{bbm}
\usepackage{comment}
\usepackage[USenglish]{babel}
\usepackage[utf8]{inputenc}
\usepackage[T1]{fontenc}

\makeatletter
\def\imod#1{\allowbreak\mkern10mu({\operator@font mod}\,\,#1)}
\makeatother

\theoremstyle{plain}
\newtheorem{thm}{Theorem}
\theoremstyle{definition}
\newtheorem{defi}[thm]{Definition}
\newtheorem{cor}[thm]{Corollary}
\newtheorem*{cor*}{Corollary}
\newtheorem*{lem*}{Lemma}
\newtheorem*{thm*}{Theorem}
\newtheorem*{defi*}{Definition}
\newtheorem{lem}[thm]{Lemma}
\newtheorem{question}[thm]{Question}

\newtheorem*{clm*}{Claim}

\newtheorem{rem}[thm]{Remark}

\renewcommand{\phi}{\varphi}

\newcommand{\concatB}{\mathbin{\rotatebox[origin=c]{90}{\scalebox{.7}{(\kern1ex)}}}}

\begin{document}

\title[Non-linear iterations and higher splitting]{Non-linear iterations and higher splitting}

\author{\"Omer Faruk Ba\u{g}}
\address{Institute of Mathematics, Kurt G\"odel Research Center, University of Vienna, Augasse 2-6, UZA 1 - Building 2, 1090 Wien, Austria}
\email{oemer.bag@univie.ac.at}

\author{Vera Fischer}
\address{Institute of Mathematics, Kurt G\"odel Research Center, University of Vienna, Augasse 2-6, UZA 1 - Building 2, 1090 Wien, Austria}
\email[Corresponding Author]{vera.fischer@univie.ac.at}

\thanks{\emph{Acknowledgments.}: The authors would like to thank the Austrian Science Fund (FWF) for the generous support through Grant Y1012-N35.}

\subjclass[2000]{03E35, 03E17}

\keywords{cardinal characteristics; eventually narrow sequences; generalised Baire spaces, large cardinals; forcing; non-linear iterations}

\begin{abstract} 
We show that generalized eventually narrow sequences on a strongly inaccessible cardinal $\kappa$ are preserved under the Cummings-Shaleh
non-linear iterations of the higher Hechler forcing on $\kappa$. 
Moreover assuming GCH, $\kappa^{<\kappa}=\kappa$, we show that
\begin{enumerate}
	\item if $\kappa$ is strongly unfoldable, $\kappa^+\leq\beta=\hbox{cf}(\beta)\leq \hbox{cf}(\delta)\leq\delta\leq\mu$ and $\hbox{cf}(\mu)>\kappa$, 
	then there is a cardinal preserving generic extension in which 
	$$\mathfrak{s}(\kappa)=\kappa^+\leq\mathfrak{b}(\kappa)=\beta\leq\mathfrak{d}(\kappa)=\delta\leq 2^\kappa=\mu.$$
	\item  if $\kappa$ is strongly inaccessible, $\lambda>\kappa^+$, then in the generic extension obtained as the $<\kappa$-support iteration of $\kappa$-Hechler forcing of length $\lambda$ there are no $\kappa$-towers of length $\lambda$.
\end{enumerate}
\end{abstract}

\maketitle
\section{Introduction}

The topic ``cardinal characteristics of the continuum'' is a broad subject, which has been studied in many research articles and surveys like \cite{blass} or \cite{vnDwn}. Combinatorial cardinal invariants, some of which build the subject of this article, give an insight to the combinatorial and topological structure of the (generalized) real line.

The splitting, bounding and dominating numbers, denoted by $\mathfrak{s}, \mathfrak{b}$ and $\mathfrak{d}$ are due to David D. Booth, Fritz F. Rothberger and Miroslav Kat\v{e}tov respectively. While
$\mathfrak{s}, \mathfrak{b} \leq \mathfrak{d}$, the characteristics $\mathfrak{s}$ and $\mathfrak{b}$ are independent. Introducing the notion of an eventually narrow sequence and showing the preservation of such sequences under the finite support iterations of Hechler poset for adjoining a dominating real, Baumgartner and Dordal showed that consistently $ \aleph_1 = \mathfrak{s} < \mathfrak{b}$. The consistency of $\mathfrak{b}=\aleph_1<\mathfrak{s}=\aleph_2$  is due to S. Shelah and is in fact the first appearance of the method of creature forcing.
Studying the existence of ultrafilters $\mathcal{U}$, which have the property that for a given unbounded family $\mathcal{H}\subseteq{^\omega\omega}$, the relativized Mathis poset $\mathbb{M}(\mathcal{U})$ preserves the family $\mathcal{H}$ unbounded 
(appearing more recently in the literature as $\mathcal{H}$-Canjar filters, see~\cite{OGDK}), the second author jointly with J. Stepr\={a}ns generalized the result to an arbitrary regular uncountable $\kappa$, i.e. showed the consistency of $\mathfrak{b}=\kappa<\mathfrak{s}=\kappa^+$ (see~\cite{VFJS}). The more general inequality of $\mathfrak{b}=\kappa < \mathfrak{s}=\lambda$ was obtained only after certain developments of the method of matrix iteration, namely the appearance of a method of preserving maximal almost disjoint families along matrix iterations introduced by the second author and J. Brendle in~\cite{JBVF}.

In strong contrast to the countable case, the generalized bounding and splitting numbers are not independent. Indeed, Raghavan and Shelah showed that  $\mathfrak{s}(\kappa) \leq \mathfrak{b}(\kappa)$ for each regular uncountable $\kappa$. Moreover, by a result of Motoyoshi, 
$\mathfrak{s}(\kappa) \geq \kappa$ if and only if $\kappa$ is strongly inaccessible. Later T. Suzuki showed that under the same assumption $\mathfrak{s}(\kappa) \geq \kappa^+$ if and only if $\kappa$ is weakly compact (see e.g. \cite{szk}). An easy diagonalization argument, shows that $\kappa^+\leq\mathfrak{b}(\kappa)$ and so unless $\kappa$ is weakly compact, $\mathfrak{s}(\kappa) < \kappa^+ \leq \mathfrak{b}(\kappa)$.

In this article we want to further address the behaviour of $\mathfrak{s}(\kappa)$ and $\mathfrak{b}(\kappa)$ in the presence of large cardinal properties on $\kappa$. Our results are to a great extent based on the preservation of a strong splitting property, namely the preservation of generalized eventually narrow sequences in the context of linear and non-linear iterations.

\begin{defi*} $ $
	\begin{enumerate}
		\item A sequence $\langle a_\xi : \xi < \lambda \rangle$, where each $a_\xi$ is in $[\kappa]^\kappa$, is eventually narrow (we also say $\kappa$-eventually narrow) if $\forall a \in [\kappa]^\kappa$ $\exists \xi < \lambda $ $\forall \eta > \xi$ $a \not\subseteq^* a_\eta$.

		\item A sequence $\langle a_\xi : \xi < \lambda \rangle$, where each $a_\xi$ is in $[\kappa]^\kappa$, is eventually splitting (we also say $\kappa$-eventually splitting) if $\forall a \in [\kappa]^\kappa$ $\exists \xi < \lambda $ $\forall \eta > \xi$ $a_\eta$ splits $a$. 
	\end{enumerate}	
\end{defi*}

Devising a special forcing notion  $D(\kappa, Q)$, which can be interpreted as the non-linear iteration of the higher Hechler forcing 
(see Definition~\ref{def_H} and Definition~\ref{def_CS}), Cummings and Shelah show that any admissible assignment to $\mathfrak{b}(\kappa),\mathfrak{d}(\kappa)$ and $2^\kappa$ for $\kappa$ regular uncountable, can be realized in a generic extension of a model of GCH (see Section \ref{prelim}). We show that if $\kappa$ is strongly inaccessible, then generalized eventually narrow sequences on $\kappa$ are preserved not only by linear iterations of the higher Hechler forcing, but also by the non-linear iterations of Cummings and Shelah, which is our main preservation result (see Theorem~\ref{thm_pres}):

\begin{thm*} Assume GCH, $\kappa^{<\kappa}=\kappa$, $\kappa$ is strongly inaccessible, $cf(\lambda) > \kappa$. If $\tau=\langle a_\xi:\xi<\lambda\rangle$ is a $\kappa$-eventually narrow sequence in $V$, then $\tau$ remains eventually narrow in  $V^{D(\kappa, Q)}$.
\end{thm*}

In analogy with~\cite{drdl} we introduce a notion of a derivative for a dense open subset of $D(\kappa, Q)$ (see Definition~\ref{def_derivatives} and Theorem~\ref{thm1}), which is a key tool in obtaining the above theorem. Imposing further large cardinal properties on $\kappa$, we come to our main result: 

\begin{thm*}
	(GCH, $\kappa^{<\kappa} = \kappa$) Assume $\kappa$ is strongly unfoldable and $\beta, \delta, \mu$ are cardinals with $\kappa^+ \leq \beta= cf(\beta) \leq cf(\delta) \leq \delta \leq \mu$ and $cf(\mu)>\kappa$; then there exists a cardinal preserving generic extension  of the ground model, where $\mathfrak{s}(\kappa) = \kappa^+ \land \mathfrak{b}(\kappa) = \beta \land \mathfrak{d}(\kappa) = \delta \land \mathfrak{c}(\kappa) = \mu$ holds.
\end{thm*}

Our techniques modify to the countable case and so in particular establish the consistency of    $\mathfrak{s}=\omega_1<\mathfrak{b}<\mathfrak{d}<\mathfrak{c}$. 
Alternatively, one can  force with $\mathbb{B}(\omega_1)$ over a model of
$\omega_1<\mathfrak{b}<\mathfrak{d}<\mathfrak{c}$. 
Note however (at least to the best knowledge of the authors) that even though there are already good higher analogues of random forcing 
(the poset for adjoining a single random real), we still do not seem to have an appropriate analogue of $\mathbb{B}(\omega_1)$ into the uncountable.

In addition, we show that in generic extensions obtained as the linear iteration of the higher Hechler poset, there are no long $\kappa$-towers. 

\begin{cor*}
(GCH, $\kappa^{<\kappa} = \kappa$) Assume $\kappa$ is strongly inaccessible. Then there is a cardinal preserving generic extension where $\mathfrak{c}(\kappa)= \kappa^{++}$ and $Spec(\mathfrak{t}(\kappa)) = \{\kappa^+\}$ hold.
\end{cor*}

Controlling $\mathfrak{s}(\kappa)$ strictly above $\kappa^+$ simultaneously with $\mathfrak{b}(\kappa),\mathfrak{d}(\kappa)$ and $2^\kappa$ remains an interesting open question. 
For a model of $\aleph_1 < \mathfrak{s} < \mathfrak{b}=\mathfrak{d} < \mathfrak{c}(=\mathfrak{a})$ see~\cite{VFDM}.

\section{Preliminaries}\label{prelim}

Now we recall some preliminaries and definitions.
\begin{defi}\label{def_s}
Let $\kappa$ be regular. Let $a$ and $b$ be elements in $[\kappa]^\kappa$.
\begin{enumerate}

\item Then $a \subseteq^* b$  holds, if $\vert a \setminus b \vert < \kappa$.

\item Further for $a, b\in [\kappa]^\kappa$, we say $a$ splits $b$ if $\vert b \setminus a \vert = \vert a \cap b \vert = \kappa$.

\item A family $\mathcal{S} \subseteq [\kappa]^\kappa$ is splitting if $\forall b \in [\kappa]^\kappa$ $\exists a \in \mathcal{S}$ such that $a$ splits $b$.

\item Finally $\mathfrak{s}(\kappa)$ denotes the generalized splitting number:

\begin{center}
$\mathfrak{s}(\kappa) = min \{ \vert \mathcal{S}\vert : \mathcal{S} \subseteq [\kappa]^\kappa, \mathcal{S}$ is $splitting\}$.
\end{center}
\end{enumerate}
\end{defi}

\begin{defi}\label{def_bd}
Let $\kappa$ be regular and let $f$ and $g$ be functions from $\kappa$ to $\kappa$, i.e. $f, g \in {}^{\kappa}\kappa$.

\begin{enumerate}

\item Then $g$ eventually dominates $f$, denoted by $f <^* g$, if $\exists \alpha < \kappa$ $\forall \beta > \alpha$ $f(\beta) < g(\beta)$.

\item A family $\mathcal{F} \subseteq {}^{\kappa}\kappa$, is dominating if $\forall g \in {}^{\kappa}\kappa $ $\exists f \in \mathcal{F}$ such that $g<^*f$.

\item A family $\mathcal{F} \subseteq {}^{\kappa}\kappa$ is unbounded if $\forall g \in {}^{\kappa}\kappa$ $\exists f \in \mathcal{F}$ such that $f \not<^*g$.

\item $\mathfrak{b}(\kappa)$ and $\mathfrak{d}(\kappa)$ denote the generalized bounding and dominating numbers respectively:

\begin{center}
$\mathfrak{b}(\kappa) = min \{ \vert \mathcal{F}\vert : \mathcal{F} \subseteq {}^{\kappa}\kappa, \mathcal{F}$ is unbounded$\}$

$\mathfrak{d}(\kappa) = min \{ \vert \mathcal{F}\vert : \mathcal{F} \subseteq {}^{\kappa}\kappa, \mathcal{F}$ is dominating$\}$.
\end{center}

\item Finally $\mathfrak{c}(\kappa) = 2^\kappa$.
\end{enumerate}
\end{defi}

In \cite{cmmngs} it is shown that $\kappa^+ \leq \mathfrak{b}(\kappa) = cf(\mathfrak{b}(\kappa)) \leq cf(\mathfrak{d}(\kappa)) \leq \mathfrak{d}(\kappa) \leq \mathfrak{c}(\kappa)$ holds.

More generally one defines bounding and dominating for arbitrary posets as follows: 

\begin{defi}[\cite{cmmngs}]
Let $(P, \leq_P)$ be a partial order. 

\begin{enumerate}

\item We call $U \subseteq P$ unbounded if $\forall p \in P \exists q \in U q \not\leq_P p$.

\item $\mathfrak{b}(P) = min \{ \vert U \vert : U \subseteq P$, $U$ is unbounded$\}$.

\item A subset $D \subseteq P$ is dominating if $\forall p \in P \exists q \in D p \leq_P q$.

\item $\mathfrak{d}(P) = min \{ \vert D \vert : D \subseteq P$, $D$ is dominating$\}$.
\end{enumerate}
\end{defi}

As mentioned above it is known that the triple $(\mathfrak{b}(\kappa), \mathfrak{d}(\kappa), \mathfrak{c}(\kappa))$ can be anything not contradicting the known results in ZFC:

\begin{thm}[\cite{cmmngs}]
(GCH at and above $\kappa$, $\kappa^{<\kappa} = \kappa$) Suppose $\beta, \delta, \mu$ are cardinals satisfying $\kappa^+ \leq \beta = cf(\beta) \leq cf(\delta) \leq \delta \leq \mu$ and $cf(\mu) > \kappa$. Then there is a cardinal preserving forcing notion $M(\kappa, \beta, \delta, \mu)$ such that $V^{M(\kappa, \beta, \delta, \mu)} \vDash \mathfrak{b}(\kappa) = \beta \land \mathfrak{d}(\kappa) = \delta \land \mu = \mathfrak{c}(\kappa)$.
\end{thm}

Although the exact definition a strong unfoldable cardinal is not strictly necessary to understand the results of this article we state it for the sake of completeness. First if $\kappa$ is strongly inaccessible, then a $\kappa$-model denotes a transitive structure $M$ of size $\kappa$, such that $M \vDash ZFC-P$, $\kappa \in M$ and $M^{<\kappa} \subseteq M$, i.e. $M$ is closed under building sequences of size less than $\kappa$.

\begin{defi}[\cite{VA}, \cite{J}] $ $

\begin{enumerate}

\item Let $\lambda$ be an ordinal. A cardinal $\kappa$ is $\lambda$-strongly unfoldable iff 
\begin{enumerate}
\item $\kappa$ is strongly inaccessible
\item for every $\kappa$-model $M$ there is an elementary embedding $j: M \to N$ with critical point $\kappa$ such that $\lambda < j(\kappa)$ and $V_\kappa \subseteq N$.
\end{enumerate}

\item A cardinal $\kappa$ is called strongly unfoldable if it is $\theta$-strongly unfoldable for every ordinal $\theta$.

\end{enumerate}
\end{defi}

A strongly unfoldable cardinal is in particular weakly compact. In the next section we will also use T. A. Johnstone's theorem concerning the indestructiblity of strongly unfoldable cardinals. 

\begin{thm}[\cite{J}]
Let $\kappa$ be strongly unfoldable. Then there is a set forcing extension where the strong unfoldability of $\kappa$ is indestructible by forcing notions of any size which are $<\kappa$-closed and have $\kappa^+$-c.c..
\end{thm}

\begin{rem}
In Theorem \ref{t1} one could make the stronger assumption of a supercompact cardinal. Then Laver preparation can be used to make the supercompactness of $\kappa$ indestructible by the $\kappa$-directed closed forcing poset $D(\kappa, Q)$. The straight forward proof of the later assertion is given for the sake of completeness.  
\end{rem}

\begin{lem}\label{dir-clo}
$D(\kappa, Q)$ is $\kappa$-directed closed.
\end{lem}

\begin{proof}
Let $W:= \{p_\alpha: \alpha < \gamma\}$ be a directed set of conditions where $\gamma < \kappa$.  We define their common extension $p$ as follows: $dom(p) = \bigcup_{\alpha < \gamma}dom(p_\alpha)$ then $|dom(p)| < \kappa$ by regularity.
For any $\alpha < \gamma$ and $b \in dom(p_\alpha)$ let $p_\alpha(b) = (t_\alpha^b, \dot{F}_\alpha^b)$. Then we define $p(b) = (t_p^b, \dot{F}_p^b)$ where $t_p^b$ is the union of the stems (since $W$ is directed): $t_p^b = \bigcup\{t_\alpha^b: b \in dom(p_\alpha)\}$ and $\dot{F}_p^b$ is a $P_b$-name for the pointwise supremum of the second coordinates $\{\dot{F}_\alpha^b: b \in dom(p_\alpha)\}$.
Now it is easy to verify by an induction on the rank of $b \in dom(p)$ (its $Q$-rank) that $p$ is a common extension for $W$. 
\end{proof}

\section{Consistency of $\mathfrak{s}(\kappa) = \kappa^+ < \mathfrak{b}(\kappa) < \mathfrak{d}(\kappa) < \mathfrak{c}(\kappa)$}\label{sbdc}

In this section we want to obtain the consistency of $ \kappa^+ = \mathfrak{s}(\kappa) < \mathfrak{b}(\kappa) < \mathfrak{d}(\kappa) < \mathfrak{c}(\kappa)$ where $\kappa$ is strongly unfoldable and $\mathfrak{b}(\kappa), \mathfrak{d}(\kappa)$ and $\mathfrak{c}(\kappa)$ can be arbitrary uncountable cardinals $\beta, \delta, \mu \geq \kappa^+$ not contradicting $\beta= cf(\beta) \leq cf(\delta) \leq \delta\leq \mu$ and $cf(\mu)>\kappa$. Unless otherwise specified, $\kappa$ is a regular uncountable cardinal.

Let $\kappa^{<\kappa} \uparrow: = \{s \in \kappa^{<\kappa}: s$ is strictly increasing$\}$ and ${}^{\kappa}\kappa \uparrow: = \{f \in {}^{\kappa}\kappa: f$ is strictly increasing$\}$. 

\begin{defi}\label{def_H}
The $\kappa$-Hechler poset is defined as the set $\mathbb{H}(\kappa) = \{(s,f): s \in \kappa^{<\kappa} \uparrow, f \in {}^\kappa \kappa \uparrow\}$ with extension relation given by $(t,g) \leq_{\mathbb{H}(\kappa)} (s,f)$ iff $$s \subseteq t \land \forall \alpha \in \kappa ~[g(\alpha)\geq f(\alpha)] \land \forall \alpha \in dom(t)\setminus dom(s)~ [t(\alpha) > f(\alpha)].$$
When $\kappa$ is clear from the context, we write just $\mathbb{H}$ instead of $\mathbb{H}(\kappa)$.
\end{defi}

Next we recall the definition of a non-linear forcing iteration $D(\kappa, Q)$ from \cite{cmmngs} with the additional assumption that the stems and second coordinates are strictly increasing. Note that the
poset is dense in the original one.

\begin{defi}\label{def_CS}
Let $(Q, \leq_{Q})$ be a well-founded poset such that $\kappa^+ \leq \mathfrak{b}((Q, \leq_{Q}))$. 
Extend $Q$ to a partial order $Q'=Q\cup\{m\}$ with a maximal element $m$.
Recursively on $Q'$, define for each $a\in Q'$ a forcing notion $P_a$ as follows:
\begin{itemize}
	\item  Fix $a\in Q'$ and suppose for each $b <_{Q'} a$ the poset $P_b$ has been defined. Then $P_a$ consists of functions $p$ such that $dom(p) \subseteq a\downarrow\; : = \{c \in Q': c <_{Q'} a\}$ and $|dom(p)| < \kappa$ and $\forall b \in dom(p)~ p(b) = (t, \dot{F})$ where $t \in \kappa^{<\kappa} \uparrow$ and $\dot{F}$ is a $P_b$-name for an element in  ${}^{\kappa}\kappa \uparrow$. 

    \item Thus for every $a \in Q'$ each $p \in P_a$ is of the form $p = ( \bar{p}_0, \bar{p}_1)$, where $\bar{p}_0 = \langle s(b) \rangle_{b \in dom(p)}$, $\bar{p}_1 = \langle \dot{F}_b \rangle_{b \in dom(p)}$ and each pair $(s(b), \dot{F}_b)$ is a $P_b$-name for a $\kappa$-Hechler condition. 

    \item For $p \in P_a$ and $c \in Q'$ let $p \upharpoonright c\downarrow = (\langle s(b)\rangle_{b \in dom(p) \cap c \downarrow}, \langle \dot{F}_b\rangle_{b \in dom(p) \cap c \downarrow})$.
\end{itemize}

The extension relation of $P_a$ is defined as follows: $p \leq q$ iff $dom(q) \subseteq dom(p)$ and
$$\forall b \in dom(q)~[ p \upharpoonright b \Vdash_{P_b} (\bar{p}_0(b), \bar{p}_1(b)) \leq_{\mathbb{H}(\kappa)} (\bar{q}_0(b), \bar{q}_1(b))].$$

Finally, let $D(\kappa, Q)= P_m$.
\end{defi}

\begin{rem}\label{rem_nonlin_it} $ $
\begin{enumerate}
\item The fact that $D(\kappa, Q)$ has the $\kappa^+$-c.c. and $\kappa$-closed is shown in \cite{cmmngs}.

\item If $\lambda$ is a regular uncountable cardinal and $(Q, \leq_Q)= (\lambda, \in)$, then $D(\kappa, Q) = D(\kappa, (\lambda,\in))$ is the $<\kappa$ support iteration of $\mathbb{H}(\kappa)$. Note also, that 
$\mathfrak{b}((\lambda, \in))= \mathfrak{d}((\lambda, \in)) = \lambda$.
\end{enumerate}
\end{rem}


Whenever $X$ and $Y$ are given sets, let\\
\centerline{$fin_{<\kappa}(X, Y)= \{f:~ f$ is a partial function from $X$ to $Y$, $|dom(f)|<\kappa\}$.}

\begin{defi}
Whenever $\bar{s} \in fin_{<\kappa}(Q, \kappa^{<\kappa} \uparrow)$, we denote by $  l_{\bar{s}} \in {}^{dom(\bar{s})}\kappa$ the lengths of the sequences in $\bar{s}$, i.e. $l_{\bar{s}} = \langle dom(\bar{s}(a))\rangle_{a \in dom(\bar{s})}$. 
\end{defi}

\begin{defi}\label{def_derivatives}
Let $D$ be open dense in $D(\kappa, Q)$, i.e. $\forall p \in D(\kappa, Q) \exists q \in D$ such that $q\leq p$ and whenever $p \in D$ and $q \leq p$ then $q \in D$. Define a sequence of subsets of $fin_{<\kappa}(Q, \kappa^{<\kappa} \uparrow)$, referred to as a sequence of derivatives, as follows:

\begin{enumerate}
\item\label{1} $D_0 = \{\bar{s} \in fin_{<\kappa}(Q, \kappa^{<\kappa} \uparrow) | ~ \exists p \in D ~[\bar{p}_0 = \bar{s}]\}$,

\item\label{2} $D_{\alpha +1} = \{\bar{s} \in fin_{<\kappa}(Q, \kappa^{<\kappa} \uparrow) |$
\begin{enumerate}
\item\label{2a} $\bar{s} \in D_\alpha~$, or

\item\label{2b} $\exists \bar{t} \in D_\alpha \exists ! a \in dom(\bar{t})\hbox{ such that } dom(\bar{s}) = dom(\bar{t}) \setminus \{a\} \land \bar{t} \upharpoonright dom(\bar{s}) = \bar{s}~$, or

\item\label{2c}
$\exists \bar{l} \in fin_{<\kappa}(Q, \kappa)~[  dom(\bar{l}) = dom(\bar{s})] \land \forall c \in dom(\bar{l})~[\bar{l}(c) \geq dom(\bar{s}(c))] \land \exists c \in dom(\bar{l})~[ \bar{l}(c) > dom(\bar{s}(c))]$ and $\exists \{\bar{t}_\beta: \beta < \kappa \} \subseteq D_\alpha \forall \beta < \kappa ~[ \bar{s} \subseteq  \bar{t}_\beta \land l_{\bar{t}_\beta \upharpoonright dom(\bar{l})} = \bar{l} \land \forall b \in dom(\bar{l})~[ \bar{t}_\beta(b)(dom(\bar{s}(b))) > \beta]] \}$, and 
\end{enumerate}

\item\label{3} $D_\alpha = \bigcup \{D_\beta \vert ~ \beta < \alpha \}$ if $\alpha$ is a limit ordinal.
\end{enumerate}
\end{defi}

Item (\ref{2b}) says that the hierarchy of derivatives is closed under shortening the domain, i.e. whenever a sequence (of sequences) $\bar{t}$ appears in the hierarchy and the sequence $\bar{s}$ is obtained from $\bar{t}$ only by forgetting points in the domain of $\bar{t}$, then $\bar{s}$ also appears in the hierarchy of derivatives (at a higher level). In item (\ref{2c}) first $\bar{l}$ fixes a sequence of lengths on a domain. Then for every $\beta \in \kappa$ a sequence of sequences $\bar{t}_\beta$ is found such that each one's domain contain the domain of $\bar{l}$ and on this domain the lengths of the sequences in $\bar{t}_\beta$ coincide with the lengths fixed by $\bar{l}$ (how each $\bar{t}_\beta$ behaves outside this domain $dom(\bar{l})$ doesn't matter). Further each sequence in $\bar{t}_\beta$ is an end-extension of the sequence in $\bar{s}$ at the same point and if the former is strictly longer than the latter, then it goes above any value on its new domain.

Due to (\ref{2a}) and (\ref{3}) this sequence is increasing, i.e. $D_\alpha \subseteq D_{\alpha +1}$. Consequently this increasing sequence of derivatives has to stabilize at some index below $|fin_{<\kappa}(Q, \kappa^{<\kappa} \uparrow)|^+$, that is there  exists $\gamma < |fin_{<\kappa}(Q, \kappa^{<\kappa} \uparrow)|^+$ such that $D_\gamma = D_{\gamma+1}$.

\begin{thm}\label{thm1}
Assume GCH, $\kappa^{<\kappa} = \kappa$ and $\kappa$ is strongly inaccessible. Let $\gamma$ be the least such that $D_\gamma = D_{\gamma+1}$. Then $D_\gamma = fin_{<\kappa}(Q, \kappa^{<\kappa} \uparrow)$.
\end{thm}

\begin{proof}
Suppose not and let $\bar{s} \in fin_{<\kappa}(Q, \kappa^{<\kappa} \uparrow)\setminus D_\gamma$. For the purposes of this proof, we will use the following notion:

\begin{defi*}
A sequence $\bar{t} \in D_\gamma$ is said to be a minimal extension of $\bar{s}$ if  $dom(\bar{t}) = dom(\bar{s}), \bar{s} \subseteq \bar{t}$ and whenever $ \bar{l} \in {}^{dom(\bar{s})} \kappa$ is such that $l_{\bar{s}} \leq \bar{l} \leq l_{\bar{t}}$ pointwise and $\exists c \in dom(\bar{s})~[ \bar{l}(c) < l_{\bar{t}}(c))]$ then $ \bar{t} \upharpoonright \bar{l}\not\in D_\gamma$, where $\bar{t} \upharpoonright \bar{l} = \langle \bar{t}(a) \upharpoonright \bar{l}(a): a \in dom(\bar{t}) \rangle$.
\end{defi*}

 For the first let us claim, that for given lengths on $dom(\bar{s})$, there are less than $\kappa$ many minimal extensions with these lengths:

\begin{clm*}
For every $\bar{l} \in {}^{dom(\bar{s})} \kappa$ we have that $|T_{\bar{l}} | < \kappa$ where 

\centerline{$T_{\bar{l}} := \{ \bar{t}~:~ \bar{t}$ is a minimal extension of $\bar{s}$ with $l_{\bar{t}} = \bar{l} \}$.}
\end{clm*}

\begin{proof}[Proof of the Claim]
Suppose not and let $\bar{l}\in^{dom(\bar{s})}\kappa$ be such that $|T_{\bar{l}}|\geq\kappa$. Then $|T_{\bar{l}}|=\kappa$, as  $T_{\bar{l}} \subseteq {}^{dom(\bar{s})}(\kappa^{<\kappa} \uparrow)$, $| \kappa^{<\kappa} \uparrow| = \kappa$ and $|dom(\bar{s})| < \kappa$.
	For each $a\in dom(\bar{l})$ and each $\bar{t}\in T_{\bar{l}}$ let
	$$\rho_a(\bar{t})=\sup\{ \bar{t}(a)(\alpha):\alpha <\bar{l}(a)\}$$
	and let $\rho_a=\sup\{\rho_a(\bar{t}): \bar{t}\in T_{\bar{l}}\}$. 
If for each $a\in dom(\bar{s})$, $\rho_a <\kappa$, then $|T_{\bar{l}}|<\kappa$. This is due to $\kappa^{<\kappa} = \kappa$, the regularity of $\kappa$, $|T_{\bar{l}}| = \kappa$ and the inaccessibility of $\kappa$. Thus, there is $a\in dom(\bar{s})$ such that $\rho_a=\kappa$. Now, if for each $\alpha <\bar{l}(a)$, $\mu_\alpha:=\sup\{\bar{t}(a)(\alpha):\bar{t}\in T_{\bar{l}}\}<\kappa$, 
then $\rho_a=\sup_{\alpha<\bar{l}(a)} \mu_\alpha < \kappa$, which is a contradiction. Therefore, there is $\alpha<\bar{l}(a)$ such that $\mu_\alpha=\kappa$. Pick $\alpha$ least such that $\mu_\alpha=\kappa$. 
Then in particular, $|\{\bar{t}(a)\upharpoonright\alpha:\bar{t}\in T_{\bar{l}}\}|<\kappa$ and so we can find $u \in {^{<\kappa}\kappa}\uparrow$ and $T'\subseteq T_{\bar{l}}$ of cardinality $\kappa$ such that for each $\bar{t}\in T'$, $\bar{t}(a)\upharpoonright \alpha=u$ and $\{\bar{t}(a)(\alpha): \bar{t}\in T'\}$ is unbounded in $\kappa$. Fix $\bar{t}\in T'$.  Then $\bar{s}'\in{^{dom(\bar{s})} (\kappa^{<\kappa}\uparrow)}$ where $\bar{s}'(a)=u$ and $\bar{s}'(b)=\bar{t}(b)$ for $b\neq a$ is an element of $D_{\gamma+1}=D_\gamma$, contradicting the minimality of $\bar{t}$. 
\end{proof}

We continue with the proof of the theorem. As there are $<\kappa$ many minimal extensions for a fixed sequence of lengths $\bar{l}$ and $|{}^{dom(\bar{s})} \kappa| = \kappa$, we can define on $dom(\bar{s})$ functions which go above all minimal extensions in $T_{\bar{l}}$. For any $\bar{l} \in {}^{dom(\bar{s})}\kappa$ let 
$$\rho_{\bar{l}}:= sup\{\bar{l}(a): a \in dom(\bar{l}) \land \bar{l}(a) \not= dom(\bar{s}(a))\}.$$ 
Since $|dom(\bar{s})| < \kappa$, $\rho_{\bar{l}} \in \kappa$ for each $\bar{l} \in {}^{dom(\bar{s})}\kappa$.  First we deal with those minimal extensions $\bar{t}$ with $\rho_{\bar{l}_{\bar{t}}} $ is a successor. Then 
for each $a \in dom(\bar{s})$ and each  $dom(\bar{s}(a)) \leq \alpha <\kappa $ the set\\ 
\centerline{$H_{a, \alpha} = \{\bar{t}(a)(\alpha): \bar{t} \in T_{\bar{l}}, \rho_{\bar{l}} = \bar{l}(a) = \alpha +1\}$}
is bounded by the above claim. Thus there are functions $g_a \in {}^\kappa\kappa\uparrow$ such that $g_a(\alpha)> sup(H_{a, \alpha})$ for each $dom(\bar{s}(a)) \leq \alpha <\kappa$.
Thus, $g_a$ dominates at $\alpha$ the values of all minimal extensions $\bar{t}$ whose maximal length (not equaling the lengths of $\bar{s}$) is the successor $\alpha+1$ which again is witnessed at point $a \in dom(\bar{t}) (= dom(\bar{s}))$.

Second we deal with those minimal extensions $\bar{t}$ with $\rho_{\bar{l}_{\bar{t}}} $ is a limit. 
For each limit $\alpha \in \kappa$ with $\alpha \geq sup\{\bar{l}_{\bar{s}}(a): a \in dom(\bar{s})\}$ and $\beta \in \alpha$ the set
$$G_\beta = \{\bar{t}(a)(\beta): \bar{t} \in T_{\bar{l}}, \rho_{\bar{l}} = \alpha\},$$ 
which is again bounded by the claim. Thus inductively for each limit $\alpha \in \kappa$ with $\alpha \geq sup\{\bar{l}_{\bar{s}}(a): a \in dom(\bar{s})\}$ and $\beta \in \alpha$ we can define a function $g \in {}^\kappa \kappa \uparrow$ such that $g(\beta)> sup(G_\beta)$ if $g(\beta)$ is not defined already.
So, $g$ dominates below an ordinal $\alpha$ suitable values of all minimal extensions whose maximal length is the limit $\alpha$.

Finally for each $a \in dom(\bar{s})$, let $\dot{f}_a$ be a $P_a$-name  for the pointwise maximum of $g_a$ and $g$. Consider the condition $p \in D(\kappa, Q)$ with $dom(p) = dom(\bar{s})$ and $\forall a \in dom(p)~[ p(a) = (\bar{s}(a), \dot{f}_a)]$. By the density of $D$ we can find a condition $q \leq p$ such that $q \in D$. So the element $ \bar{t} \in {{}^{dom(q)}(\kappa^{<\kappa} \uparrow)}$ with $\bar{t} = \bar{q}_0$ is in $D_0$ and $\bar{s} \subseteq \bar{t}$ (by the extension relation). 
Then, some initial segment $\bar{t}'$ of $\bar{t}$ must be a minimal extension of $\bar{s}$. If $\bar{t}' \upharpoonright dom(\bar{s}) = \bar{s}$, then $\bar{s} \in D_\gamma$, which is a contradiction. Otherwise $\exists \bar{l}' \in {}^{dom(\bar{s})}\kappa \exists a \in dom(\bar{l}')~[\bar{l}'(a) > dom(\bar{s}(a))]$ and $l_{\bar{t}'} = \bar{l}'$. Let $\lambda$ be $sup\{\bar{l}'(b)|~ \bar{l}'(b) \not = dom(\bar{s}(b))\}$. If $\lambda = \alpha +1$ and $a \in dom(\bar{l}')~ [\alpha +1 = \bar{l}'(a)]$, then $\bar{t}(a)(\alpha) = \bar{t}'(a)(\alpha) < \dot{g}_a(\alpha)\leq \dot{f}_a(\alpha)$ which is a contradiction to $q \leq p$. Suppose $\lambda$ is a limit and $a \in dom(\bar{l}')$ with $\bar{l}'(a) > dom(\bar{s}(a))$. Take $\beta \in \bar{l}'(a)$ with $\beta \geq dom(\bar{s}(a))$. Then $\bar{t}(a)(\beta) = \bar{t}'(a)(\beta) < \dot{g}(\beta) \leq \dot{f}_a(\beta)$ which is a contradiction to $q \leq p$.
\end{proof}

\begin{defi}\label{def_evna} $   $
	\begin{enumerate}
		\item A sequence $\langle a_\xi : \xi < \lambda \rangle$, where each $a_\xi$ is in $[\kappa]^\kappa$, is eventually splitting if $\forall a \in [\kappa]^\kappa$ $\exists \xi < \lambda $ $\forall \eta > \xi$ $a_\eta$ splits $a$. 
		\item A sequence $\langle a_\xi : \xi < \lambda \rangle$, where each $a_\xi$ is in $[\kappa]^\kappa$, is eventually narrow if $\forall a \in [\kappa]^\kappa$ $\exists \xi < \lambda $ $\forall \eta > \xi$ $a \not\subseteq^* a_\eta$.
	\end{enumerate}	
\end{defi}

Note that $\tau = \langle a_\xi : \xi < \lambda \rangle$ is eventually splitting iff the sequence $\tau' = \langle b_\xi : \xi < \lambda \rangle$, defined as $b_{2\xi}= a_\xi$ and $b_{2\xi+1}= \kappa \setminus a_\xi$, is eventually narrow.

\begin{thm}\label{thm_pres}
Assume GCH, $\kappa^{<\kappa} = \kappa$, $\kappa$ is strongly inaccessible and let $cf(\lambda) > \kappa$. Then any eventually narrow sequence $\tau= \langle a_\xi: \xi < \lambda \rangle$ remains eventually narrow in  $V^{D(\kappa, Q)}$.
\end{thm}
\begin{proof}
Suppose not. Fix $p \in D(\kappa, Q)$ and a name $\dot{a}$ for a subset of $\kappa$ of size $\kappa$ such that
$$p \Vdash_{D(\kappa, Q)} \forall \xi < \lambda \exists \eta > \xi ~[\dot{a} \subseteq^* a_\eta].$$
Let $\varrho$ be a regular cardinal such that $D(\kappa, Q) \in H(\varrho) = \{x \in WF: \vert trcl(x)\vert < \varrho\}$. Let $\mathcal{N}$ be an elementary substructure of $H(\varrho)$ of size $\kappa$ such that $D(\kappa, Q) \in \mathcal{N}$, $\dot{a} \in \mathcal{N}$ and $\dot{f}^p_a \in \mathcal{N}$, where we denote $p(a)= (s^p_a, \dot{f}^p_a)$. Since $\tau$ is eventually narrow, for every $a\in [\kappa]^\kappa \cap \mathcal{N}$ there is a $\xi < \lambda$ such that for all $\eta > \xi$ we have $a \not\subseteq^* a_\eta$.  However $|\mathcal{N}| = \kappa$, $\tau$ is of length $\lambda$ and $cf(\lambda) > \kappa$, so this will yield $\kappa$-many $\xi$'s smaller than $\lambda$; so we can not reach $\lambda$ in $\kappa$-many steps. Hence $\exists \xi' < \lambda ~ \forall c \in \mathcal{N} \cap [\kappa]^\kappa~ \forall \eta' \geq \xi' ~ [c \not\subseteq^* a_{\eta'}]$.

Since $p \Vdash ``\forall \xi < \lambda \exists \eta > \xi~ \dot{a} \setminus a_\eta$ is of size less than $\kappa$'', in particular $p$ forces the existence of an ${\eta_0}$ greater than $\xi'$ (fixed two lines above) such that $\dot{a} \setminus a_{\eta_0}$ is of size less than $\kappa$. By extending $p$ (the extension is also called $p$) we have that there is an $\alpha_0 \in \kappa$ and $\eta_0 > \xi'$ such that $p \Vdash \forall j \geq \alpha_0$ if $j \in \dot{a}$ then $j \in a_{\eta_0}$.

Let $\dot{h}$ be a $D(\kappa, Q)$-name such that $\Vdash `` \dot{h} \hbox{ enumerates }\dot{a}"$. Then, in particular $\Vdash \forall \zeta < \kappa~ \dot{h}(\zeta) \geq \zeta$. To define $\dot{h}$ we only used $\dot{a}$ which was in $\mathcal{N}$ and $\dot{h} \in \mathcal{N}$ as well. For the purposes of this proof, we will use the following notions:
\begin{defi*} $  $
	\begin{enumerate}
    \item Let $u=(\bar{u}_0, \bar{u}_1) \in D(\kappa, Q)$. A sequence $\bar{t} \in fin_{<\kappa}(Q, \kappa^{<\kappa} \uparrow)$ is said to be $u$-admissible if $\bar{u}_0 \subseteq \bar{t}$ and $(\bar{t}\upharpoonright dom(u), \bar{u}_1)$ is a condition in $D(\kappa, Q)$.
    \item 
    Let $\bar{t} \in fin_{<\kappa}(Q, \kappa^{<\kappa} \uparrow)$ and let $\bar{\tau} = \langle \dot{g}_a: a \in dom(\bar{t}) \rangle$ where $\forall a \in dom(\bar{t})~ \dot{g}_a$ is a $P_a$-name for an element in ${}^\kappa\kappa\uparrow$. We say that $\bar{\tau}$ is $\bar{t}$-admissible if $q(\bar{t},\bar{\tau}) = \langle (\bar{t}(a), \dot{g}_a): a \in dom(\bar{t})\rangle$ is a condition in $D(\kappa, Q)$.
    \end{enumerate}
\end{defi*}

\begin{clm*}
Let $\bar{t} \in fin_{<\kappa}(Q, \kappa^{<\kappa} \uparrow)$ be $p$-admissible and $\zeta \geq \alpha_0$. Then $Z_{\bar{t}}(\zeta) \not= \emptyset$, where 
$$Z_{\bar{t}}(\zeta)= \{j: \forall \bar{\tau} [\bar{\tau}\hbox{ is } \bar{t}\hbox{-admissible}\to \exists r \leq_{D(\kappa, Q)} q(\bar{t}, \bar{\tau})\hbox{  such that }r \Vdash \dot{h}(\zeta) = j] \}.$$
\end{clm*}
\begin{proof}
Fix $\zeta \geq \alpha_0$ and let $D = \{u \in D(\kappa, Q): \exists j \in \kappa~[u \Vdash \dot{h}(\zeta) = j]\}$. Then $D$ is dense, open and
we can form the sequence of derivatives $\langle D_\alpha \rangle_{\alpha \leq \gamma}$ where $\gamma$ is the least with  $D_\gamma = D_{\gamma+1} =fin_{<\kappa}(Q, \kappa^{<\kappa} \uparrow)$. We  will prove the claim inductively on $\alpha \leq \gamma$ for all $p$-admissible $\bar{t} $.

If $\bar{t} \in D_0$ we have $\exists u \in D\hbox{ such that } \bar{u}_0 = \bar{t}$ and $\exists j~ [u \Vdash \dot{h}(\zeta) = j]$. Let $\bar{\tau}$ be $\bar{t}$-admissible. Then $q(\bar{t}, \bar{\tau})$ and $u~ (=(\bar{t}, \bar{u}_1))$ are compatible with common extension $r$. Thus $r \Vdash \dot{h}(\zeta) = j$ and so $Z_{\bar{t}}(\zeta) \not= \emptyset$. For limit ordinals $\alpha$ the claim is true by the induction hypothesis, since $D_\alpha = \bigcup \{D_\beta: \beta < \alpha\}$. Let $\bar{t} \in D_{\alpha +1}\setminus D_\alpha$ be $p$-admissible. By definition of $D_{\alpha+1}$ there are two possibilities:

First $\exists \bar{t}' \in D_\alpha \exists ! b \in Q: dom(\bar{t}') = dom(\bar{t}) \cup \{b\} \land \bar{t}' \upharpoonright dom(\bar{t}')\setminus \{b\} = \bar{t}~$. Since $\bar{t}$ is $p$-admissible, $\bar{t}'$ is also $p$-admissible (easily seen by definition) and by induction hypothesis $Z_{\bar{t}'}(\zeta) \not= \emptyset$. That is for some $j_0 \in \kappa$, we have\\ 
\centerline{$\forall \bar{\tau}' [\bar{\tau}'$ is $\bar{t}'$-admissible $\to \exists r \leq_{D(\kappa, Q)} q(\bar{t}', \bar{\tau}')$ such that $r \Vdash \dot{h}(\zeta) = j_0]$.$\hspace{0.9 cm} (\ast)$}
We claim that $j_0 \in Z_{\bar{t}}(\zeta)$. Indeed consider any $\bar{t}$-admissible $\bar{\tau}$. Then $\tau$ can be extended to a $\bar{t}'$-admissible $\bar{\tau}'$. Then $q(\bar{t}', \bar{\tau}') \leq q(\bar{t}, \bar{\tau})$ and by $(\ast)$, there is $r \leq_{D(\kappa, Q)} q(\bar{t}', \bar{\tau}')$ with $r \Vdash \dot{h}(\zeta) = j_0$. Then by transitivity $r \leq_{D(\kappa, Q)} q(\bar{t}, \bar{\tau})$ and we conclude that $Z_{\bar{t}}(\zeta) \not= \emptyset$.

Second there is a sequence $\langle \bar{t}_\beta : \beta \in \kappa \rangle$ of elements of $D_\alpha$ such that $\forall \beta < \kappa: dom(\bar{t}) \subseteq dom(\bar{t}_\beta)$ and $l_{\bar{t}_\beta \upharpoonright dom(\bar{t}) }  = \bar{l}$ (for some $\bar{l} \in {}^{dom(\bar{t})}\kappa$) and $\forall b \in dom(\bar{t})~[\bar{t}_\beta(b)(dom(\bar{t}(b))) > \beta]$. Since such a sequence exists in $H(\varrho)$ and the latter was an existential statement and $\mathcal{N} \preccurlyeq H(\varrho)$, by the  Tarski-Vaught-Criterion we can find a witness in $\mathcal{N}$. So assume $\langle \bar{t}_\beta : \beta \in \kappa \rangle \in \mathcal{N}$.

At this point we distinguish between two either-or cases. Case 1: There is a $j \in \kappa$ such that $j \in Z_{\bar{t}_\beta}(\zeta)$ for $\kappa$-many $\beta$. Let $\bar{\tau} = \langle \dot{g}_a: a  \in dom(\bar{t})\rangle$ be $\bar{t}$-admissible and for each $\beta< \kappa$ let $\bar{\tau}^\beta$ be $\bar{t}_\beta$-admissible with $\bar{\tau}^\beta\upharpoonright dom(\bar{t})= \bar{\tau}$. We have that ``$\exists r \leq q(\bar{t}_\beta, \bar{\tau}^\beta) ~[r \Vdash \dot{h}(\zeta) = j]$'' for $\kappa$-many $\bar{t}_\beta$'s, but not all of these $q(\bar{t}_\beta, \bar{\tau}^\beta)$ extend $q(\bar{t}, \bar{\tau})$. However since we have $\kappa$-many such $\bar{t}_\beta$'s and $\forall b \in dom(\bar{t}):~\bar{t}_\beta(b)(dom(\bar{t}(b))) > \beta$ we can find one (actually infinitely many) $q(\bar{t}_\beta, \bar{\tau}^\beta) \leq q(\bar{t}, \bar{\tau})$ and consequently infinitely many $r \leq q(\bar{t}, \bar{\tau})$ such that $j \in Z_{\bar{t}_\beta}(\zeta)$; hence $j \in Z_{\bar{t}}(\zeta) \not = \emptyset$. (To find a $q(\bar{t}_\beta, \bar{\tau}^\beta)$ as desired we choose $\beta$ such that $\beta > \bigcup\{\dot{g}_a(\bar{l}(a)): a \in dom(q_g)\}$. Then for such a $\beta$ and any $a \in dom(\bar{t})$ and $\alpha$ with $dom(\bar{t}(a))\leq \alpha < \bar{l}(a) : \big[ \bar{t}_\beta(a)(\alpha) > \beta > \dot{g}_a(\bar{l}(a)) > \dot{g}_a(\alpha)) \big]$).

Case 2: Fix by the induction hypothesis one $j_\beta \in Z_{\bar{t}_\beta}(\zeta)$ (e.g. choose the minimal one) and consider the set $J:=\{ j_\beta: \beta \in \kappa\}$. This set is of size $\kappa$, because otherwise it would have an upper bound in $\kappa$, so $\exists \alpha_0 < \kappa ~\forall \alpha, \beta \geq \alpha_0:~ j_\alpha =j_\beta$. But then we would have a $j$ which is in all $Z_{\bar{t}_\beta}(\zeta)$'s for $\beta \geq \alpha_0$, so we would have a $j$ which is in $\kappa$-many $Z_{\bar{t}_\beta}(\zeta)$'s, which is in fact Case1. So $|J| = \kappa$, but $J$ consists of $j_\beta$'s which are elements of $Z_{\bar{t}_\beta}(\zeta)$ and these were defined using $\dot{h}$ which was in $\mathcal{N}$ and the sequence  $\langle \bar{t}_\beta : \beta \in \kappa \rangle$ which was also in $\mathcal{N}$, so we may take $J \in \mathcal{N}$. This further means that $|J\setminus a_{\eta_0}| = \kappa$. So choose $\beta$ large enough such that $j_\beta \geq \alpha_0,~ \beta \geq \bigcup\{\dot{f}_a^p(\bar{l}(a)): a \in dom(p)\}$ and $j_\beta \not\in a_{\eta_0}$. Then for this particular $\beta$ we have $u \leq v \leq p$ where $\bar{v}_0 = \bar{t}$ and $\bar{v}_1 \upharpoonright dom(p)=  \bar{p}_1$ and $\bar{u}_0 = \bar{t}_\beta$ and $\bar{u}_1 \upharpoonright dom(p)=  \bar{p}_1$. For the first extension relation note that for $a  \in dom(v)$ we have $\bar{t}_\beta(a)(dom(\bar{t}(a))) > \beta \geq \dot{f}_a^p(\bar{l}(a))$ so this extension really holds. And since $j_\beta \in Z_{\bar{t}_\beta}(\zeta)$ there is by the definition of $Z_{\bar{t}_\beta}(\zeta)$ some $r_\beta \leq u$ such that $r_\beta \Vdash \dot{h}(\zeta) = j_\beta$. But then since $j_\beta \geq \alpha_0$ and $r_\beta \Vdash ``\forall j \geq \alpha_0$ if $j \in \dot{a}$ then $j \in a_{\eta_0}$'' ($p$ forced this). All together we have $j_\beta \in a_{\eta_0}$ which is a contradiction.
\end{proof}

By the claim $Z_{\bar{s}}(\zeta) \not= \emptyset$ for $\zeta \geq \alpha_0$ where  $\bar{s} \in {}^{dom(p)}(\kappa^{<\kappa} \uparrow)$ with $\bar{s} = \bar{p}_0$ since $p \leq p$. Choose $k_\zeta \in Z_{\bar{s}}(\zeta)$ for each $\zeta \geq \alpha_0$ and consider the set $K:=\{k_\zeta: ~\zeta \geq \alpha_0\}$. Since $\Vdash \dot{h}(\zeta) \geq \zeta$ we have $k_\zeta \geq \zeta$ for all $\zeta$, hence $K$ is of size $\kappa$. Since $K$ is definable from $\bar{s}$ and other parameters of $Z_{\bar{s}}(\zeta)$, we have $K \in \mathcal{N}$, so $K\setminus a_{\eta_0}$ has size $\kappa$. Now let $k_\zeta \in K \setminus a_{\eta_0}$ be chosen; so by definition $\exists r \leq p$ such that $r \Vdash \dot{h}(\zeta) = k_\zeta$ and again $r \Vdash ``\forall j \geq \alpha_0$ if $j \in \dot{a}$ then $j \in a_{\eta_0}$'', and $k_\zeta \geq \zeta \geq \alpha_0$ so we have $k_\zeta \in a_{\eta_0}$ which is a contradiction.
\end{proof}

Finally we formulate the theorem:

\begin{thm}\label{t1}
(GCH, $\kappa^{<\kappa} = \kappa$) Assume $\kappa$ is strongly unfoldable and $\beta, \delta, \mu$ are cardinals with $\kappa^+ \leq \beta= cf(\beta) \leq cf(\delta) \leq \delta \leq \mu$ and $cf(\mu)>\kappa$; then there exists a forcing poset $\mathbb{P}_{\kappa, \beta, \delta, \mu}$ such that the cardinal preserving generic extension $V^{\mathbb{P}_{\kappa, \beta, \delta, \mu}}$ satisfies $$\mathfrak{s}(\kappa) = \kappa^+ \land \mathfrak{b}(\kappa) = \beta \land \mathfrak{d}(\kappa) = \delta \land \mathfrak{c}(\kappa) = \mu.$$ 
\end{thm}
\begin{proof}
One part follows as in \cite{cmmngs}: In the ground model $V_0$ let $Q$ be a poset with $\mathfrak{b}((Q, \leq_q)) = \beta \leq \mathfrak{d}((Q, \leq_q)) = \delta$ and let $Q'$ be a cofinal well-founded subset of $Q$ ($Q'$ has the same bounding and dominating numbers).
Next we construct a poset which consists of a copy of $(\mu, \in)$ at the bottom and a cofinal copy of $Q'$ at the top; so let $R$ consist of pairs $(p, i)$ such that either $ i = 0 \land p \in \mu$ or $i = 1 \land p \in Q'$. The order relation is defined as $(p,i) \leq (q,j)$ iff $i = 0 \land j = 1$ or $i = j = 1 \land p \leq_{Q'} q$ or $i = j= 0 \land p \leq q$ in $\mu$. Finally let $\mathbb{P}_{\kappa, \beta, \delta, \mu} = D(\kappa, R)$. This forcing poset is $\kappa$-closed and has the $\kappa^+$-c.c., so we start the whole construction by Lottery preparation $\mathbb{Q}$ and let $V = V_0^{\mathbb{Q}}$. Lottery preparation preserves GCH for all cardinals $\geq \kappa$, which suffices to apply Theorem 2 in \cite{cmmngs}.  As shown in \cite{cmmngs} this forcing poset satisfies $V^{\mathbb{P}_{\kappa, \beta, \delta, \mu}} \vDash \beta = \mathfrak{b}(\kappa) \leq \delta = \mathfrak{d}(\kappa) \leq \mu = \mathfrak{c}(\kappa)$. 

Since $\kappa$ is strongly unfoldable and $\mathbb{P}_{\kappa, \beta, \delta, \mu}$ is $\kappa$-closed and has the $\kappa^+$-c.c., we may assume that $\kappa$ is still strongly unfoldable in the generic extension after doing Lottery preparation, so $\mathfrak{s}(\kappa) \geq \kappa^+$. Further the last theorem shows that there is a splitting family of size $\kappa^+$ in the extension: Every forcing adding a dominating real, also adds a splitting real. The first $\kappa^+$-many splitting reals added by the first $\kappa^+$-many steps of the Hechler iteration, build an eventually splitting sequence in the intermediate model $V^{P_{ \kappa^+}}$, which is preserved as a such in the final model. Hence $V^{\mathbb{P}_{\kappa, \beta, \delta, \mu}} \vDash \kappa^+ = \mathfrak{s}(\kappa)$.
\end{proof}

\begin{rem}
Thus under the assumption that there is a strongly unfoldable $\kappa$, it is consistent that all four characteristics are different, i.e.  $\kappa^+ = \mathfrak{s}(\kappa) < \beta = \mathfrak{b}(\kappa) < \delta = \mathfrak{d}(\kappa) < \mu = \mathfrak{c}(\kappa$).
\end{rem}

\section{Consistency of $Spec(\mathfrak{t}(\kappa)) = \{\kappa^+\} \land \mathfrak{c}(\kappa) = \kappa^{++}$}\label{no_long_towers}

\begin{defi}[\cite{grt}]\label{def_tower}
Let $\kappa$ be a regular uncountable cardinal.
\begin{enumerate}	
\item  A sequence $\langle a_\xi: \xi < \mu \rangle$ of elements in $[\lambda]^\lambda$ is a descending $\subseteq^*$-sequence  if for $\xi < \eta < \mu$ we have that $a_\eta \subseteq^* a_\xi$.

\item A family or sequence of subsets of $\lambda$ has the strong intersection property (SIP) if any subfamily of size less than $\lambda$ has intersection of size $\lambda$.

\item A $\lambda$-tower is a descending $\subseteq^*$-sequence with the SIP and no pseudo- intersection of size $\lambda$, in other words $\forall a \in [\omega]^\omega ~\exists \xi < \mu~ a \not\subseteq^* a_\xi$. 

\item Let the tower number $\mathfrak{t}(\lambda)$ denote the minimal cardinality of a $\lambda$-tower.

\item Finally we define the spectrum of $\lambda$-towers as $Spec(\mathfrak{t}(\lambda)) = \{|\tau|: \tau$ is a $\lambda$-tower$\}$.
\end{enumerate}
\end{defi}

By a diagonal argument one can find a pseudo- intersection for any family $\mathcal{F} \subseteq [\lambda]^\lambda$ with the SIP and $|\mathcal{F}| \leq \lambda$, hence $\lambda^+ \leq \mathfrak{t}(\lambda)$.

\begin{rem}
Note that a $\lambda$-tower is eventually narrow: Let $\langle a_\xi: \xi < \mu \rangle$ be a $\lambda$-tower and $a \in [\lambda]^\lambda$ be arbitrary. Since $\langle a_\xi: \xi < \mu \rangle$ is a $\lambda$-tower $\exists \xi < \mu~ |a \setminus a_\xi|= \lambda$. Let $\xi' > \xi$, since $\langle a_\xi: \xi < \mu \rangle$ is a descending $\subseteq^*$-sequence we have that $a_{\xi'} \subseteq^* a_\xi$. Now $a \setminus a_{\xi'} \supseteq [(a \setminus a_\xi)\setminus a_{\xi'}] \cup [(a \cap a_\xi)\setminus a_{\xi'}]$ $= [(a \setminus a_\xi)\setminus (a_{\xi'} \setminus a_\xi)] \cup [(a \cap a_\xi)\setminus a_{\xi'}]$, but recall that $(a \setminus a_\xi)$ was of size $\lambda$ and $|(a_{\xi'} \setminus a_\xi)|<\lambda$, so $a \setminus a_{\xi'}$ contains a subset of size $\lambda$, so itself has size $\lambda$. Since $\xi'$ was arbitrary all together we have $\forall a \in [\lambda]^\lambda ~\exists \xi < \mu~ \forall \xi' > \xi~ |a \setminus a_{\xi'}| = \lambda$, so the $\lambda$-tower is eventually narrow.
\end{rem}

Along a non-linear iteration of Hechler forcings, as defined in Definition \ref{def_CS}, also witnesses for $\mathfrak{t}(\kappa)$ are preserved; so are also the witnesses for other higher analogues of  $\mathfrak{t}$, e.g. $\mathfrak{t}_{cl}(\kappa)$ or $\mathfrak{t}^*(\kappa)$ (see \cite{fmss} for the definitions). Next we show that it is consistent that the generalized cardinal invariant $\mathfrak{t}(\kappa)$ has no witness of size $\mathfrak{c}(\kappa)$. Here we use a well-ordered $(\lambda, \in)$ in place of the well-founded poset $(Q, \leq_Q)$ as in Remark \ref{rem_nonlin_it} (2). 

\begin{thm}\label{notower}
Assume GCH, $\kappa^{<\kappa} = \kappa$ and $\kappa$ is strongly inaccessible. Let $\lambda > \kappa^+$ be a regular uncountable cardinal. Then in $V^{D(\kappa, (\lambda, \in))}$ there are no $\kappa$-towers of length $\lambda$, but there are descending $\subseteq^*$-sequences of length $\lambda$ with the SIP.
\end{thm}

\begin{proof}
Suppose $\tau = \langle a_\xi : \xi < \lambda \rangle$ is a $\kappa$-tower in $V^{D(\kappa, (\lambda, \in))}$. Then by the definition of a $\kappa$-tower we have $\forall a \in [\kappa]^\kappa \cap V^{D(\kappa, (\lambda, \in))} ~ \exists \xi = \xi(a) < \lambda ~ a \not\subseteq^* a_\xi$. Since GCH holds it is easily observed that the following holds: 

\begin{clm*}
$\forall \alpha < \lambda$ we have $| \mathcal{P}(\kappa) \cap V^{P_\alpha}| < \lambda$.
\end{clm*} 
Consider the function $f$, where $f(\alpha) = sup\{\xi(a): a \in \mathcal{P}(\kappa) \cap V^{P_\alpha}\}$. By the above claim and the regularity of $\lambda$ we have $f(\alpha) < \lambda$, i.e. $f \in {}^{\lambda}\lambda \cap V^{D(\kappa, (\lambda, \in))}$. By the Approximation Lemma (\cite[Lemma IV.7.8.]{kun})  (remember that $D(\kappa, (\lambda, \in))$ has the $\kappa^+$-c.c. and $\kappa^+ <\lambda$ is regular) we can find a function $g: \lambda \to \lambda$ in the ground model $V$ such that $ \forall \alpha < \lambda ~ f(\alpha)\leq g(\alpha)$.

The set $M= \{\gamma < \lambda: \langle a_\xi: \xi < \gamma \rangle \in V^{P_\gamma}\}$ contains a club $C$ in $\lambda$ (as the fixed points of a normal function form a club). 
By a well-known argument (\cite[Lemma III.6.13]{kun}) we have that $D:=\{\alpha < \lambda : \forall \beta < \alpha~ g(\beta) < \alpha\}$ is a club in $\lambda$, so the intersection $C \cap D$ is also a club. Since $E^{\kappa^+}_\lambda = \{\alpha < \lambda: cf(\alpha) = \kappa^+\}$ is a stationary set, there is $\gamma$ such that $\langle a_\xi: \xi < \gamma\rangle \in V^{P_\gamma}$, $\forall \alpha < \gamma ~g(\alpha)<\gamma$ and $cf(\gamma) = \kappa^+ > \kappa$. 
Then $\langle a_\xi: \xi < \gamma \rangle$ is a $\kappa$-tower in $V^{P_\gamma}$. Indeed, if $a \in [\kappa]^\kappa \cap V^{P_\gamma}$ then $a$ is already added at stage $\alpha$, for some $\alpha < \gamma$ (since $cf(\gamma)> \kappa$ and such stages do not add new $\kappa$-reals). Now $|a_{f(\alpha)} \setminus a| = \kappa$,  $f(\alpha)\leq g(\alpha)<\gamma$ and so $\langle a_\xi: \xi < \gamma \rangle$ is a $\kappa$-tower in $V^{P_\gamma}$.  Hence it is also a $\kappa$-tower in $V^{D(\kappa, (\lambda, \in))}$ because $\kappa$-towers are preserved and $V^{D(\kappa, (\lambda, \in))}$ is obtained by iterated forcing over $V^{P_\gamma}$. This yields a contradiction, since $a_\delta$ for $\gamma<\delta<\lambda$ is almost contained in each $a_\beta$ for $\beta < \gamma$.

For the last statement of the theorem, consider the reals $f_\xi \in {}^\kappa \kappa$ where $f_\xi$ is the dominating real added at stage $\xi$ of the iteration. So if $\xi < \eta < \lambda$ we have $f_\xi <^*f_\eta$. Define each $c_\xi$ as $c_\xi = \{(\alpha,\beta) \in  \kappa \times \kappa:~ \beta \geq f_\xi(\alpha) \}$ and consider the sequence $\langle c_\xi: \xi < \lambda \rangle$. Using a bijection between $\kappa$ and $\kappa \times \kappa$ it is easily seen that this is a descending $\subseteq^*$-sequence with the SIP.
\end{proof}

\begin{cor}\label{cor_towers}
Assume GCH, $\kappa^{<\kappa} = \kappa$ and $\kappa$ is strongly inaccessible. Then there is a cardinal preserving extension where $\mathfrak{c}(\kappa)= \kappa^{++}$ and $Spec(\mathfrak{t}(\kappa)) = \{\kappa^+\}$ hold.
\end{cor}
\begin{proof}
It suffices to choose $\lambda= \kappa^{++}$ in Theorem \ref{notower}. Then 
$V^{D(\kappa, (\kappa^{++}, \in))} \vDash Spec(\mathfrak{t}(\kappa)) = \{\kappa^+\}$ as there are no long $\kappa$-towers and $\mathfrak{t}(\kappa) \geq \kappa^+$. $\mathfrak{c}(\kappa) \geq \kappa^{++}$ is witnessed by the $\kappa^{++}$-many Hechler $\kappa$-reals added. $\mathfrak{c}(\kappa) \leq \kappa^{++}$ is seen by the standard argument of counting nice names (see~\cite{kun}). 
\end{proof}

\section{ Questions}\label{questions_section}

The method in Section \ref{sbdc} allows $\mathfrak{s}(\kappa)$ to be $\kappa^+$. However generalizing the methods in \cite{drdl} seems not to be enough to control $\mathfrak{s}(\kappa)$ arbitrarily above $\kappa^+$.

\begin{question}
Is it consistent that $\kappa^+ <\mathfrak{s}(\kappa)<\mathfrak{b}(\kappa)<\mathfrak{d}(\kappa)<2^\kappa$ for $\kappa$ at least weakly compact?
\end{question}

\end{document}